\documentclass[a4paper,10pt]{article}

\title{Verbally closed subgroups of free groups}
\author{A. Myasnikov, V. Roman'kov}

\date{}

\usepackage{amsmath,amsthm,amsfonts}
\usepackage{graphicx}
\usepackage{hyperref}
\usepackage[usenames]{color}

\newtheorem{theorem}{Theorem}[section]
\newtheorem{lemma}[theorem]{Lemma}
\theoremstyle{definition}

\newtheorem{definition}[theorem]{Definition}

\newtheorem{problem}[theorem]{Problem}

\newtheorem{proposition}[theorem]{Proposition}

\def\CK{\mathcal K}

\newcounter{comcount}

\begin{document}

\maketitle

\begin{abstract}
We prove that every verbally closed subgroup of a free group $F$ of a finite rank is a retract  of $F.$ 
\end{abstract}

%\tableofcontents

\section{Introduction}
\label{se:intro}

Algebraically closed objects play an extremely important part in modern algebra.    In this paper we study verbally closed and algebraically closed subgroups of free groups. 

Recall, that if $\CK$ is a class of structures in a language $L$ then a structure $A \in \mathcal{K}$ is called {\em algebraically closed} in $\CK$ if for any positive existential sentence  $\phi(x_1, \ldots,x_n)$ in the language $L$  with constants from $A$ if $\phi$ holds in some $B \in \CK$ that contains $A$  then it holds in $A$.  We refer  to \cite{Hodges2} for general facts on algebraically closed structures.
An interesting particular case occurs when $A$ is algebraically closed in $\CK = \{A,B\}$  for some $B$ containing $A$ as a substructure $A \leq B$. In this event $A$ is termed  {\em algebraically closed} in $B$.  Another typical and useful variation on algebraically closed structures appears when one restricts the  definition above onto  $\phi$ from  a fixed subset $\Phi$ of positive existential sentences from $L$, in which case one gets $\Phi$-algebraically closed structures. One the other hand, if instead of positive existential sentences  one has in the definitions above arbitrary existential sentences $\phi$ then this defines an {\em existentially closed} structures in $\CK$.

In the case of groups the notions above can be explained in pure algebraic terms. To this end we remind some  terminology.  For groups $H$ and $G$ we write  $H \leq G$  if $H$ is a subgroup of $G$ and refer to this as  an {\em extension} of $H$ to $G$. Let $X = \{x_1, x_2, ..., x_k, ... \}$ be countable infinite  set of variables, and $F(X)$ be the free group with basis $X.$ By an equation with variables $x_1, ..., x_n \in X$ and constants from $H$ we mean an arbitrary expression $E(x_1, ..., x_n, H) = 1$ where $E(x_1, ..., x_n, H)$ is a word in the alphabet $X^{\pm 1} \cup H,$ in other words $E(x_1, ..., x_n, H)$ lies in $F[H] = F(X) \ast H,$ the free product of $F(X)$ and $H.$ In the case when the left side of the equation does not depend from $H$ we will omit $H$ in its expression.   We say that $E(x_1, \ldots,x_n, H) = 1$ has a solution in $G$ if there is a substitution 
$x_i \to g_i$ for some elements $g_i \in G$ such that $E(g_1, \ldots, g_n) = 1$ in $G$.  It is easy to see that a subgroup $H$ is
algebraically closed in $G$ if  and only if for every finite system of equations $S = \{E_i(x_1, \ldots,x_n, H) = 1 \mid i = 1, \ldots, m\}$ with constants from  $H$ the following holds: if $S$  has a solution in $G$ then it has a solution in $H$.  By the same token, a group $H$ is  algebraically closed in a class of groups $\CK$ if and only if $H$ is algebraically closed in every extension $H \leq G$ with $G \in \CK$. Replacing systems of equations in the definitions above by systems of equations and inequalities with coefficients in $H$ one gets the notion of existentially closed groups in $\CK$, as well as all corresponding variations.

  Groups algebraically (existentially) closed in the class of all groups   were introduced by Scott in \cite{Scott}.  They have been thoroughly
studied in 1970-80's, see, for example,   papers by Macintyre \cite{Mac}, Eklof and  Sabbagh \cite{ES}, Belegradek \cite{B1,B2},  Ziegler \cite{Z}; and books by Hodges \cite{Hodges1} and Higman and Scott \cite{HS}. Nowadays, a lot more is known about groups algebraically or existentially closed in various specific classes of groups $\CK$, in particular, when $\CK$ consists of various  nilpotent, solvable, or locally finite groups. For details  we refer to a survey by Leinen \cite{Leinen}.  On the other hand,  not much is known about algebraically or existentially closed groups in the classes of groups with some presence of negative curvature, for example, in the classes of groups universally equivalent to a given hyperbolic group. To this end we would like to mention a work by Jaligot and Ould Houcine \cite{JH} on existentially closed CSA-groups  (see \cite{MR} for definitions and various properties of CSA groups). Notice, that groups universally equivalent to a given torsion-free hyperbolic group are CSA.

 Our interest to this topic is twofold.  The first part comes from studying  Krull dimension and Cantor-Bendixon rank of groups.
  To explain, recall first that a  subgroup $H$ of a group $G$  is called a  {\em retract } of $G$, if
there is a homomorphism (termed {\em  retraction}) $\phi: G \to H$ which is identical on $H$.
In Section \ref{se:prelim}, Proposition \ref{pr:1},  we show (and it is easy)   that every retract of $G$ is algebraically closed in $G$. Furthermore, if   $G$ is finitely presented and $H$ is finitely generated then the converse is also true.  This result still holds for finitely generated $G$ which are {\em equationally Noetherian} (for definition see   Section \ref{se:prelim} below). 
However, to characterize   existentially closed subgroups one needs a stronger condition. Namely, an extension $H \leq G$ is called {\em discriminating } if
 for every finite subset  $K\subseteq G$ there is a retraction $\phi:G \to H$ such that the restriction of $\phi$ onto $K$ is injective. 
 It is easy to see again  that if $H \leq G$ is discriminating   then the subgroup $H$ is
existentially  closed in $G$; and furthermore, if 
  $G$ is finitely generated relative to $H$ and  $H$ is
equationally Noetherian then the converse is also true (Proposition \ref{pr:1b}).  If  a group $G$ is equationally Noetherian then Zariski topology on its Cartesian product  (affine space) $G^k$, defined by algebraic sets as a pre-basis of closed sets, is Noetherian \cite{BMR1}.  It was shown in \cite{MRns} that in this case the Zariski dimension of irreducible algebraic sets $Y$ in  $G^k$  is equal to the  Krull dimension of their coordinate groups $G_Y$.   Here the {\em Krull dimension} of $G_Y$ is defined (as usual) as the supremum of   the lengths of  chains 
$p_0 \subset p_1 \subset \ldots  \subset p_k$
of distinct prime ideals $p_i$ in $G_Y$, where a prime ideal in $G_Y$ is a normal subgroup $N$ of $G_Y$ such that   $N \cap G = 1$ (the subgroup $G$ naturally embeds into $G_Y$ and hence into $G_Y/N$) and   $G \leq G_Y/N$  is  a discriminating extension.

Another part of our interest in various versions of algebraic closures  comes from research on verbal width (or length) of elements in groups. To explain we need some notation. 
For  $w= w(x_1, \ldots,x_n) \in F(X)$ and a group $G$ by $w[G]$ we denote the set of all $w$-elements in $G$, i.e., $w[G] = \{w(g_1, \ldots,g_n) \mid g_1, \ldots, g_n \in G\}$.  The {\em verbal subgroup} $w(G)$ is the subgroup of $G$ generated by $w[G]$.
The {\em $w$-width} (or {\em $w$-length}) $l_{w}(g) = l_{w, G}(g)$  of an element $g \in w(G)$ is the  minimal natural number $n$ such that $g$ is a product of  $n$ $w$-elements in $G$ or their inverses; the width of $w(G)$ is the supremum of widths of its elements.    Usually, it  is  very hard to compute the $w$-length of a given  element $g \in w(G)$ or the width of $w(G)$. 
The first question of this type  goes back to the Ore's paper \cite{Ore} where he asked whether  the commutator length (i.e., the $[x,y]$-length) of every element in a non-abelian finite simple group is equal to $1$ ({\em Ore Conjecture}). Only recently the conjecture was established by  Liebeck, O'Brian, Shalev and  Tiep \cite{LBST}. For recent spectacular results on the $w$-length in finite simple groups,  we refer to the papers \cite{LS}, \cite{Shalev} and a  book \cite{Segal}. For instance, A. Shalev \cite{Shalev} proved that for any nontrivial word $w,$ every element of every sufficiently large finite simple group is a product of three values of $w.$ 

Two important questions arise naturally for an extension $H \leq G$ and a  given word $w \in F(X)$:
\begin{itemize}
\item when it is true that  $w(H) = w(G) \cap H$ or $w[H] = w[G] \cap H$?
\item when $l_{w, G}(h) = l_{w,H}(h)$ for a given $h \in w(H)$? 
\end{itemize} 

  To approach these questions we introduce a new notion of {\em verbally closed} subgroups.
  \begin{definition}
  
A subgroup $H$ of $G$ is called  {\em verbally closed} if for any word $w \in F(X)$ and $h \in H$ an equation $w(x_1, \ldots, x_n) = h$ has a solution in $G$ if and only if it has a solution in $H$, i.e., $w[H] = w[G] \cap H$ for every $w \in F(X)$.
  \end{definition} 

Notice, that verbally closed subgroups fit in the general picture of  algebraic closures, where the closure operator is defined by the set $\Phi$ of all single equations of the type $w(x_1, \ldots, x_n) = h$, where $w \in F(X)$ and $h \in H$. In general, single equations  do not suffice to get the standard algebraic  closures in groups (see examples  in the class of 2-nilpotent torsion-free groups due to Baumslag and Levin \cite{BL}).

Not much is known in general about verbally closed subgroups of a given group $G$. For instance, the following basic questions are still open for most non-abelian groups: 
\begin{itemize}
\item Is there an algebraic description of verbally closed subgroups of $G$? 
\item Is  the class of verbally closed subgroups of $G$ closed under intersections? 
\item Does there exist the verbal closure $vcl(H)$  of a given subgroup $H$ of $G$? Here $vcl(H)$ is the least (relative to inclusion) verbally closed subgroup of $G$ containing $H$.  
\item If $H$ is a finitely generated subgroup of $G$ is $vcl(H)$ (if it exists) also finitely generated?
\item Given $H \leq G$ can one find the generators of $vcl(H)$ (if it exists) effectively?
\end{itemize}

In this papers we address all the questions above in the case of  a free group $G$.
In Section \ref{se:retracts}  we prove the main result  of the paper that answers (for free groups) to the first question above: 

\medskip
\noindent
{\bf Theorem 1.} {\em  Let $F$ be a  free group of a finite rank. Then for a subgroup $H$  of $F$ the following conditions are equivalent:
\begin{enumerate}
\item [a)] $H$ is a retract of $F$.
\item [b)] $H$ is a verbally closed  subgroup of $F$.
\item [c)] $H$ is an algebraically closed subgroup of $F$.
\end{enumerate}
}

\medskip
This result clarifies the nature of verbally or algebraically closed subgroups in $F$. Surprisingly, the "weak" verbal closure operator in this case is as strong as  the standard one. Since quite a lot is known about retracts of a free group one can now easily derive some corollaries of the main result.
The proof of the theorem is rather short, but it is based on several deep known results. Firstly we use the fact, due to Lee,  that every non-abelian free group of finite rank has $C$-test words \cite{Lee}. Secondly, precise values of  the commutator verbal length of the derived subgroups in free nilpotent groups play an important part here.

In Section \ref{se:corollary} we study verbal (= algebraic) closures of subgroups  in  a given nonabelian free group $F_r$ of rank $r.$ It immediately follows from Theorem 1 that verbally (algebraically)  closed subgroup of $F_r$  are finitely generated. Furthermore,  the intersection of an arbitrary family of verbally  (algebraically) closed subgroups in $F_r$   is again verbally (algebraically) closed (see Proposition \ref{pr:intersection}), which proves the following theorem.

\medskip
\noindent
{\bf Theorem 2.} {\em 
Let $H$ be a subgroup  of  a free group $F_r$ with basis $\{f_1, ..., f_r\}.$ . Then there exists a unique minimal (with respect to inclusion) verbally closed subgroup $vcl(H)$ of $F_r$ containing $H.$ The subgroup $vcl(H)$  is also the unique minimal algebraically closed subgroup in $F_r$ containing $H.$
}

\medskip
Observe, that free factors  of $F_r$ are, of course,  retracts, but the converse is not true. Particular series of such examples (with some other interesting properties) are constructed by Martino and Ventura \cite{MV} and Ciobanu and Dicks \cite{CD}.

At the end of the section we study some algorithmic questions related to verbal closures in free groups.
The main results are collected in the following theorem.

\medskip
\noindent
{\bf Theorem 3.} {\em
Let $F_r$ be a free group with basis $\{f_1, ..., f_r\}.$  Then the following holds:
\begin{itemize}
\item [a)] There is an algorithm to decide if  a given finitely generated subgroup of $F_r$ is verbally (algebraically) closed or not.
\item [b)] There is an algorithm to construct $vcl(H)$, i.e., to find a basis of $vcl(H)$ for a given finitely generated subgroup $H$ of $F_r.$
\end{itemize}
}

\medskip
 We note, in passing, that  Diekert, Gutierrez,  and Hagenah gave an algorithm to solve equations with rational constraints  in free
 groups \cite{Diekert}, so  given an extension $H \leq F,$ where $F$ is a free group of a finite rank, one can check algorithmically whether or not an
 equation $E(x_1, \ldots, x_n, F) = 1$ (with
 coefficients in $F$)  has a solution in $F$, provided some  fixed distinguished variables satisfy an extra requirement $x_i \in H$. This gives a
 useful complementary tool to deal with algorithmic problems  related to $H.$

Recently, in \cite{JH} Ould Houcine and  Vallino  studied another notion of an algebraic closure of a subset $A$ of a group $G$, which is reminiscent to adding roots of a polynomial in one variable in a field.  In this case, an element $b$ is termed {\em algebraic} over $A$ if there is a formula $\phi(x)$ of group language such that $\phi(b)$ holds in $G$ and there are only finitely many other elements in $G$ satisfying $\phi$.
The set $ac(A)$ of all algebraic over $A$ elements forms a subgroup of $G$. How much this subgroup relates to the algebraic or verbal closure of $A$ - is not clear. However, there is one common component in all the variations of algebraic closures  discussed here - all of them form algebraic extensions in the sense of \cite{KM,MVW}.  By definition  subgroups $H \leq K$   of a free group $F$  form an {\em algebraic extension} if $H$ is not a subgroup of a proper free factor of $K$, i.e., there is no "purely transcendenyal" non-trivial extension over $H$ in $K$.    Since every finitely generated subgroup of $F$ has only finitely many such algebraic extensions  and one can find all of them effectively, this gives an approach to algorithmic problems for all types of algebraic closures and extensions  in free groups.

At the end of the paper (Section \ref{se:problems}) we discuss some related open problems.

\section{Preliminaries}
\label{se:prelim}

In this section we collect some known or simple facts on verbally, algebraic or existentially closed subgroups of groups. 

 At the beginning we mention a few simple, but useful general results. Recall that a group $G$ is called {\em equationally Noetherian} if for any $n$ every system of equations in $n$ variables with coefficients from $G$ is equivalent (has the same solution set in $G$)  to  some finite subsystem of itself \cite{BMR1,BMRom}.
 
 \begin{definition}
  An extension $H \leq G$ is called {\em discriminating } if
 for every finite subset  $K\subseteq G$ there is a retraction $\phi:G \to H$ such that the restriction of $\phi$ onto $K$ is injective. 
\end{definition}

\begin{proposition}
\label{pr:1} Let $H \leq G$ be a group  extension. Then the following holds:
\begin{itemize}
\item[1)] If $H$ is a retract of $G$ then $H$ is algebraically closed in $G.$
\item [2)]  Suppose $G$ is finitely presented and $H$ is finitely generated. Then $H$ is algebraically closed in $G$ if and only if  $H$ is a retract.
\item [3)] Suppose $G$  and $H$ are  finitely generated and $H$ is 
equationally Noetherian. Then $H$ is algebraically closed in $G$ if and only if  $H$ is a retract.
\end{itemize}
\end{proposition}
\begin{proof}
Let $\pi: G \rightarrow H$ be a retraction.  Then if a finite system of equations 
 $\Phi (x_1, \ldots, x_n, H)$ holds in $G$ on elements $g_1, \ldots,
 g_n$ then  $\Phi (x_1, \ldots, x_n, H)$ holds in $H$ on elements
 $\pi(g_1), \ldots, \pi(g_n)$, which  proves 1).
 
 To prove  2) assume that  $H$ is generated by a finite set $h_1, \ldots, h_m$ and $G$ has a finite presentation 
$G = \langle a_1, ..., a_n \mid r_1, \ldots,r_s\rangle$.   For $i = 1, ..., m$  fix a presentation $h_i = v_i(a_1, ..., a_n)$  
of $h_i$ as a word in the generators of $G$.  Then the  system of equations  
\begin{equation}\label{eq:system}
\begin{array}{l}
 v_1(x_1, ..., x_n) =h_1,  \ldots ,  v_m(x_1, ..., x_n) = h_m,  \\

\medskip
  r_1(x_1, ..., x_n) = 1, \ldots, r_s(x_1, ..., x_n) =1 
\end{array}
\end{equation}
with constants $h_i \in H$ and variables $x_1, ..., x_n$  has a solution $x_1 = a_1, \ldots, x_n = a_n$ in $G$.  Hence it has a solution $x_1 = b_1, \ldots, x_n = b_n$  in $H$.  Now, the map $a_1 \to b_1, \ldots a_n \to b_n$ defines a retraction of  $G$ onto $H,$ as claimed.

 3) is similar to 2) (see also the argument in the proof of 2) in Proposition \ref{pr:1b}).

\end{proof}

\begin{proposition}
\label{pr:1b} Let $H \leq G$ be a group  extension. Then the following holds:
\begin{itemize}
\item [1)]  If $H \leq G$ is discriminating   then the subgroup $H$ is
existentially  closed in $G.$
\item [2)] Suppose that  $G$ is finitely generated relative to $H$ and  $H$ is
equationally Noetherian. Then $H$ is existentially closed in $G$ if and only if the extension $H \leq G$ is discriminating.
\end{itemize}
\end{proposition}

\begin{proof}
Let $H \leq G$ be a  discriminating extension.  Suppose some elements  $a_1, \ldots,
 a_n \in G$  satisfy in $G$ a given finite  system $\Psi(x_1, \ldots, x_n, H)$ of equations and inequalities
 with constants from $H.$ Then there is a retraction $\pi: G \rightarrow H$ 
 such that $\pi(a_1), \ldots, \pi(a_n)$
satisfy precisely the same systems of equations and inequalities,
i.e.,  $\Psi(x_1, \ldots, x_n, H)$ holds in $H$ on  $\pi(a_1), \ldots, \pi(a_n)$.  This proves 1).

2) was  proven in  \cite{BMR1}, but we give a quick sketch of the proof here.  Let $B_n =  \{b_1, \ldots, b_n\}$ be a finite generating set of $G$ relative to $H.$ Denote by $R = R(b_1, \ldots, b_n, H) = 1$ a set of defining relations of $G$ relative to $B_n \cup H.$ One may correspond to this set  a system of equations $\Phi = \Phi (x_1, \ldots, x_n, H) $ in variables $x_1, \ldots, x_n$ and constants from $H$.  Since $H$ is equationally Noetherian the system $\Phi $ is equivalent in $H$ to some finite subsystem, say $\Phi_0 $. A given finite system  $\Psi $ of inequalities in $G$ can be rewritten into an equivalent finite system $\Psi_0$ of inequalities in $X_n \cup H.$ Since $H$ is existentially closed in $G$ the finite system of equations and inequalities $\Phi_0 \cup \Psi_0$ has a solution in $H.$ This solution gives a retraction $G \to H$ which discriminates a given finite set of elements in $G$ (which was encoded in the system $\Psi_0.$)

\end{proof}

\begin{lemma}
\label{le:1}
All types of extensions introduced above are
transitive, i.e., every chain of  extensions of a given type
results in an extension of the same type.
\end{lemma}
\begin{proof} Directly from the definitions.

\end{proof}

\section{Description of verbally (algebraically) closed subgroups of  free groups}
\label{se:retracts}

We start with several remarks. A subgroup $R$ of $G$ is a retract if  and only if  it has a normal complement $N$ in $G,$ i.e.  a normal subgroup $N$ of $G$ such that $G = RN$ and $R \cap N = 1.$ 
It is easy to see that every direct or free factor of $G$ is a retract. In particular, the trivial subgroup of $G$ is a retract.

An element $a$ of a free abelian group $A_n$ with basis $\{a_1, ..., a_n\}$  is called {\em primitive} if it can be included into some basis of $A_n.$ It is known that $a = a_1^{k_1} ... a_n^{k_n}$, where $ k_1, ..., k_n \in \mathbb{Z}$,  is primitive if and only if $gcd(k_1, ..., k_n) = 1$.

\begin{lemma}
\label{le:cycret}
Let $F_r$ be a free group of rank $r,$ and $H = gp(h)$  is a cyclic subgroup of $F_r$ generated by a non-trivial element $h \in F_r$. Then the following conditions are equivalent:
\begin{itemize}
\item [1)] $H$ is verbally closed in $F_r$;
\item [2)] $H$ is a retract of $F_r$;
\item [3)] the image of $h$ in the abelianization $F_r/[F_r,F_r]$ is primitive. 
\end{itemize}
\end{lemma}

\begin{proof}
Let $\{f_1, ..., f_r\}$ be a basis of $F_r$.
The element $h \in F_r$ can be expressed uniquely in the form
\begin{equation}
\label{eq:elfr} 
h = f_1^{k_1} ... f_r^{k_r}h'(f_1, \ldots,f_r), 
\end{equation}
where $ k_1, ..., k_r \in \mathbb{Z}$ and $ h'(f_1, \ldots,f_r)$ is a product of commutators of  words in $f_1, \ldots,f_n.$ 

To show that 1) $\rightarrow$ 3) assume that $h$   has a non primitive image  in $ F_r/[F_r,F_r]$, i.e.,   either $h \in [F_r,F_r]$   or $gcd(k_1, ..., k_n) = d > 1.$

Suppose first that  $h \in [F_r,F_r]$, so $ k_1 = \ldots =  k_r = 0$.  Replacing  each $f_i$ by  a new variable $x_i$ in (\ref{eq:elfr}) one gets an equation $h = x_1^{k_1} ... x_r^{k_r}h'(x_1, \ldots,x_r)$, with $h$ as a constant from $H$,  which has a solution in $F_r$.  However,  this equation does not have a solution in $H$, since $H$ is abelian, so $h'(h_1, \ldots,h_r) = 1$ for any $h_1, \ldots h_r \in H$. This shows that $H$ is not verbally closed in $F_r$ - contradiction. So $h \not \in [F_r,F_r]$.  Then in this case 
 $gcd(k_1, ..., k_r) = d > 1.$ The equation 
 $$h = x_1^{k_1} ... x_r^{k_r}h'(x_1, \ldots,x_r)$$
  still has a solution in $F_r$, but for any $h_1, \ldots h_r \in H$ one has 
  $$h_1^{k_1} ... h_r^{k_r}h'(h_1, \ldots,h_r) = h_1^{k_1} ... h_n^{k_r}  = h^{ds} \neq h,$$
  for some $s \in \mathbb{Z}$. Hence, the equation does not have a solution in $H$, so $H$ is not verbally closed - contradiction.  This proves  
1) $\rightarrow$ 3).

  To show that 3) $\rightarrow$ 2) assume that $h$
is  primitive. Then  there are integers $l_1, ..., l_r$ such that $k_1l_1 + ... k_rl_r = 1.$ Now we  define a homomorphism  $\varphi : F_r \rightarrow H = gp(h)$ by putting $\varphi (f_i) = h^{l_i}$ for $i = 1, ..., r.$ 
Since $H$ is abelian $\varphi(h') = 1$, so $\varphi(h) = h$ and $\varphi$ is a retraction. Hence $H$ is a retract, as claimed.

 2) $\rightarrow$ 1)   follows from Proposition \ref{pr:1} statement 1). 

\end{proof}

Below we denote by  $N_{rc} = F_r/\gamma_{c+1}F_r$ a free nilpotent group of rank $r$ and class $c.$ As usual $\gamma_{l}G$ 
denote the $l$th member of the lower central series of a group  $G.$

\begin{proposition}
\label{pr:vcret}
Every verbally closed subgroup  $H$ of a free group $F_r$ has rank at most $r$.
\end{proposition}

\begin{proof}
Suppose $H$ is a verbally closed subgroup of $F_r$ of rank $m > r$, so $H \simeq F_m.$  Consider a free nilpotent group $N_{m3} \simeq H/\gamma_4H = F_m/\gamma_4F_m$ of rank $m$ and class $3$.

 It is known (see for instance \cite{Segal}, Corollary 1.2.6) that every element $g$ in the derived subgroup $[N_{rc},N_{rc}]$ of 
a free nilpotent group $N_{rc}$ can be written as a product of $r$ commutators. More precisely, if $\{z_1, ..., z_r\}$ is a basis of $N_{rc}$ then there are elements $g_1, \ldots, g_r \in N_{rc}$ such that 
\begin{equation}
\label{eq:commform1}
g = [g_1, z_1] ... [g_r, z_r].
\end{equation}
Allambergenov and Roman'kov proved in \cite{AR} that in the case when $r \geq 2$ and $c \geq 3$ there is an element $u_r $ in $[N_{rc},N_{rc}]$ which is not equal to any product of $r-1$ commutators in $N_{rc}$.

 Now we pick  an element $u_m \in [N_{m3}, N_{m3}]$  which can not be expressed as a product of $m-1$ commutators in $N_{m3} $.   Recall that  
$N_{m3} \simeq H/\gamma_4H.$   Denote by  $h$  a preimage of $u_m$ in $H,$ notice that $h \in [H,H]$. Since  $H \leq F_r$ and $[H,H] \leq [F_r,F_r]$ it follows from (\ref{eq:commform1}) that the element $h$ can be presented in the form

\begin{equation}
\label{eq:commform2}
h = [g_1, g_1'] \ldots [g_{r}, g_{r}']f',
\end{equation}

\noindent
where 
$g_1, g_1', ..., g_{r}, g_{r}'  \in F_r$  and $f' \in \gamma_{4}F_r.$ Replace every element $g_i, g_i^\prime, f^\prime$ by the corresponding  product $g_i(f_1, \ldots, f_r), g_i^\prime(f_1, \ldots, f_r), f^\prime(f_1, \ldots, f_r)$  of  elements from a fixed basis $\{f_1, \ldots, f_r\}$  of $F_r$. The resulting equality

$$
h = [g_1(y_1, \ldots, y_r), g_1'(y_1, \ldots, y_r)] \ldots [g_{r}(y_1, \ldots, y_r), g_{r}'(y_1, \ldots, y_r)]f'(y_1, \ldots, y_r),
$$
viewed as a system in  variables $y_1, \ldots, y_r$ and a constant $h \in H$, has a solution in $F_r$, hence in $H$. 
It follows that in $N_{m3} \simeq H/\gamma_4H$ the element $h$ can be expressed as a product of $r$ commutators. Since $r < m$ we get a contradiction with our choice of  $u_m$ and $h.$
This  proves the proposition. 

 \end{proof}

Let $r \geq 2$. A non-empty word $w(z_1, ..., z_m)$ is called a {\it C-test word} in $m$ letters for $F_r$ if for any two tuples $(g_1, ..., g_m)$ and $(v_1, ..., v_m)$ of elements of $F_r$ the following holds: if  $w(g_1, ..., g_m) = w(v_1, ..., v_m) \not= 1$  then there is an element $s \in F_r$ such that $s^{-1}g_is = v_i, i = 1, \ldots, m$.  In \cite{Ivanov} Ivanov introduced and constructed first C-test words for $F_r$ in $m$ letters for any $r \geq 2$.

 In \cite{Lee} Lee constructed  for each $r, m \geq 2,$ a C-test word 
$w_r(z_1, ..., z_m)$  for $F_r$ with the additional property that $w_r(g_1, ..., g_m) = 1$ if and only if the subgroup of $F_r$ generated by $g_1, ..., g_m$ is cyclic.  We will refer to such words as {\em Lee words} for $F_r$.

\begin{theorem}
\label{th:vcisret}
Every verbally closed subgroup $H$ of a free group $F_r$ is a retract in $F_r.$
\end{theorem}

\begin{proof}
Let $H$ be a verbally closed subgroup of $F_r$. The case $r = 1$ is taken care of in Lemma \ref{le:cycret},  so we assume that $r \geq 2$.
By Proposition \ref{pr:vcret} $H$ is finitely generated with basis, say $h_1, ..., h_m$, where $m \leq r$.  For $m = 1$ the statement of the theorem   follows from Lemma \ref{le:cycret}. For the rest of proof we assume that $m \geq 2.$ 

Let $\{f_1, ..., f_r\}$ be a basis of $F_r.$  For $i = 1, ..., m$  fix a presentation $h_i = v_i(f_1, ..., f_r)$  of $h_i$ as a word in the generators.  To construct a retraction $F_r \to H$ we modify the argument in the proof of 2) in Proposition \ref{pr:1}. 

Let $w_m(z_1, ..., z_m)$ be a Lee  word for $F_r$ (for instance, constructed by Lee in \cite{Lee}).  An equation  

\begin{equation}
\label{eq:Lee}
w_m(v_1(x_1, ..., x_r), ..., v_m(x_1, ..., x_r)) = w_m(h_1, ..., h_m),
\end{equation} 
in variables $x_1, ..., x_n$ and constants $h_1, ..., h_m$ has a solution  $x_1 = f_1, \ldots, x_r = f_r$ in $F_r.$   Since $H$ is verbally closed there is a solution $x_i = g_i$ of  (\ref{eq:Lee}) with $g_i \in H$ for $i = 1, \ldots,r$, so 

\begin{equation}
\label{eq:Lee2}
w_m(v_1(g_1, ..., g_r), ..., v_m(g_1, ..., g_r)) = w_m(h_1, ..., h_m).
\end{equation} 

 Notice that the rank of $H = \langle h_1, ..., h_m \rangle$ is at least 2, so by Lee's theorem there is an element $u \in F_r$ such that 

\begin{equation}
\label{eq:con}
v_i(g_1, ..., g_r) = u^{-1}h_iu 
\end{equation}
for $ i = 1, ..., m$. 
Therefore
$$
w_m(u^{-1}h_1u, ..., u^{-1}h_mu) =u^{-1} w_m(h_1, ..., h_m)u = w_m(h_1, ..., h_m),
$$
so $u$ commutes with $h = w_m(h_1, ..., h_m)$. It follows that there is $f \in F_r$ such that $u = f^s, h = f^t$, for some $s, t \in \mathbb{Z}$. Since an equation $h = y^t$, where $y$ is a variable and $h \in H$ is a constant,  has a solution in $F_r$ it follows that it has a solution in $H$. But extraction of roots is unique in free groups, so $f \in H$ and $u = f^s \in H$.  Now, the  equality (\ref{eq:con}) implies that
$$
v_i(ug_1u^{-1}, ..., ug_ru^{-1}) = h_i,
$$
for all $i = 1, \ldots, m$.   This shows that  a homomorphism from $F_r$ to $H$ defined by $f_i \to ug_iu^{-1}, i = 1, \ldots,m$ is a  retraction of $F_r$ onto $H$. This proves the theorem.

\end{proof}

\medskip
\noindent
{\em Proof of Theorem 1.}  By Proposition \ref{pr:1} every retract in $F_r$ is algebraically closed in $F_r$, so a) $\Longrightarrow$ c). Implication c) $\Longrightarrow $ b) is obvious. Now  Theorem \ref{th:vcisret} implies b) $\Longrightarrow $ a). Hence all the conditions in Theorem  1 are equivalent, as claimed.

\section{Verbal closures of finitely generated subgroups of free groups}
\label{se:corollary}

Let $F_r$ be a free  group of finite rank $r$. In \cite{Berg} Bergman proved  that the intersection of two retracts in  $F_r$ is itself a retract.  
From this it is not hard to derive that the intersection of an arbitrary collection of retracts in $F_r$ is again a retract (see \cite[Lemma 18]{Turner} or \cite[Proposition 4.1]{MVW}). This together with Theorem \ref{th:vcisret} implies the following result.

\begin{proposition} \label{pr:intersection}
 The intersection of an arbitrary family of verbally (algebraically)  closed subgroups of $F_r$ is verbally closed. 
\end{proposition}

This proves Theorem 2. Theorem 3 follows from the propositions below.

\begin{proposition}\label{decide-retracts}
There is an algorithm to decide if  a given finitely generated subgroup of $F_r$ is verbally (algebraically) closed or not. 
\end{proposition}

\begin{proof}
 In the view of Theorem \ref{th:vcisret} it suffices to have an algorithm that decides if a given finitely generated subgroup $H$ of $F_r$ is a retract or not. Such an algorithm  has been known in folklore for some time. The formal description of an algorithm is given in \cite[Proposition 4.6]{MVW}.  For completeness we give a brief description of the algorithm here.
 
 Suppose that $F_r$ is a free group with basis $\{f_{1},\ldots,f_{r}\}$ and let
$h_{1},\ldots, h_{m}$ be a basis of $H$.  Suppose $h_i = v_i(f_1, \ldots,f_r)$ is a presentation of $h_i, i = 1, \ldots, m$, as a word in the  generators. Then $H$ is a retract of
$F_r$ if and only if there exist $x_{1},\ldots,x_{r}\in H$ such
that the endomorphism $\phi$ of $F_r$ defined by $\phi(f_{i}) =
x_{i}$ maps $H$ identically to itself.  That is, if
\begin{equation}\label{eq:retract-system}
h_i = v_{i}(x_{1},\ldots,x_{r})
\end{equation}

\noindent for $i = 1,\ldots,m$.  To decide if such $\phi$ exists or not it suffices to solve (\ref{eq:retract-system}), viewed as a system of equations in variables $x_1, \ldots, x_n$ and constants $h_1, \ldots, h_m$, in the free group $H$.  This is decidable by Makanin's algorithm
\cite{Mak}. This proves the result.
\end{proof}

\begin{proposition}\label{th:compute-closure}
There is an algorithm to  find a basis of $vcl(H)$   for a given finitely generated subgroup $H$ of $F_r$.
\end{proposition}
\begin{proof}
By Theorem \ref{th:vcisret} it suffices to construct the unique minimal retract in $F_r$ containing $H$. 
This is done in \cite[Proposition 4.5]{MVW}. 

\end{proof}

\section{Some open problems}
\label{se:problems}

\begin{problem}
What are verbally closed subgroup of a free nilpotent group of finite rank?
\end{problem}

\begin{problem}
Prove that  verbally closed subgroup of a torsion-free hyperbolic group are retracts.
\end{problem}

\medskip 
\noindent
{\em Acknowledgments.} The authors are grateful to V. Shpilrain for his helpful discussion.

\end{document}